\def\version{05/12/2022 -- version 2
\hfill\href{https://arxiv.org/abs/2205.09541}{arXiv:2205.09541}
}

\documentclass[11pt,a4paper]{amsart}%,draft

\usepackage{amssymb,amscd,amsxtra,upref}
\usepackage{fullpage}
\usepackage{hyperref}

\usepackage{kpfonts}

\usepackage[scr,scaled=1.1]{rsfso}
\usepackage{mathrsfs}

\usepackage[alphabetic,lite]{amsrefs}%

\usepackage{xcolor}
\usepackage[all,poly,web,color,matrix,arrow]{xy}
\newdir{ >}{{}*!/-8pt/@{>}}
\usepackage{rotating}

\usepackage{euflag}

\usepackage{mathrsfs}
\usepackage{stmaryrd}

\usepackage{leftidx}

\usepackage{mathtools}
\usepackage{relsize}

\usepackage{tikz}
\usetikzlibrary{arrows,calc}
\usepackage{tikz-cd}[/tikz/commutative diagrams/math mode=1]
\usepackage{tikz-network}

\usepackage{xkvltxp}
\usepackage[nomargin,inline]{fixme}
\fxusetheme{color}
\definecolor{fxnote}{rgb}{0.0000,0.6000,0.0000}

\newtheorem{thm}{Theorem}[section]
\newtheorem{lem}[thm]{Lemma}
\newtheorem{prop}[thm]{Proposition}
\newtheorem{cor}[thm]{Corollary}

\theoremstyle{remark}

\newtheorem{rem}[thm]{Remark}

\theoremstyle{definition}
\newtheorem{assump}[thm]{Assumption}

\newtheorem{defn}[thm]{Definition}
\newtheorem{examp}[thm]{Example}

\numberwithin{equation}{section}

\def\leq{\leqslant}
\def\geq{\geqslant}
\def\iso{\cong}

\def\Lra{\Longrightarrow}

\def\({\left(}
\def\){\right)}
\def\<{\left<}
\def\>{\right>}
\def\:{\colon}
\def\.{\cdot}

\def\bar#1{\overline{#1}}
\def\ph#1{\phantom{#1}}
\def\phi{{\varphi}}
\def\epsilon{{\varepsilon}}

\def\F{\mathbb{F}}

\def\Z{\mathbb{Z}}
\def\k{\Bbbk}

\def\Mod{\mathbf{Mod}}

\DeclareMathOperator{\Cohom}{Cohom}
\DeclareMathOperator{\Coext}{Coext}
\DeclareMathOperator{\Ext}{Ext}
\DeclareMathOperator{\Hom}{Hom}

\DeclareMathOperator{\im}{im}

\DeclareMathOperator{\Cotor}{Cotor}
\DeclareMathOperator{\Tor}{Tor}

\def\kO{{k\mathrm{O}}}
\def\kU{{k\mathrm{U}}}
\def\tmf{\mathrm{tmf}}

\def\StA{\mathcal{A}}
\def\StE{\mathcal{E}}
\def\StP{\mathcal{P}}
\DeclareMathOperator{\Sq}{Sq}

\def\op{{\mathrm{op}}}
\DeclareMathOperator{\Prim}{Prim}

\DeclareMathOperator{\Id}{Id}
\def\ev{\mathrm{ev}}
\DeclareMathOperator*{\colim}{colim}
\def\Palgebra{$P$-algebra}
\def\Palgebras{$P$-algebras}
\DeclareMathOperator{\pd}{pd}
\def\Mod{\mathbf{Mod}}
\def\MOD{\mathbf{Mod}}
\def\Comod{\mathbf{Comod}}

\def\rlim^#1_#2{{\lim_{#2}}^{#1}}

\newdir{ >}{{}*!/-8pt/@{>}}
\title[$P$-algebras and their duals]
{On $P$-algebras and their duals}
\date{last updated \version}
\author{Andrew Baker}
\address{
School of Mathematics \& Statistics,
University of Glasgow, Glasgow G12~8QQ, Scotland.}
\email{andrew.j.baker@glasgow.ac.uk}
\urladdr{\quad http://www.maths.gla.ac.uk/$\sim$ajb \quad https://orcid.org/0000-0002-9369-7702}
\thanks{
I would like to thank the following: The
Max-Planck-Institut f\"ur Mathematik in
Bonn for supporting my visit in January
2020 and April-May 2022; Tobias Barthel,
Ken Brown, Bob Bruner, Doug Ravenel, Nige
Ray, John Rognes and Chuck Weibel for
numerous helpful conversations on algebra
and topology.
}
\keywords{Steenrod algebra, Hopf algebras, Homotopy theory}
\subjclass[2020]{Primary 55S10; Secondary 55P42, 57T05}

\begin{document}

\begin{abstract}
The notion of $P$-algebra due to Margolis,
building on work of Moore \& Peterson,
was motivated by the case of the Steenrod
algebra at a prime and its modules. We
develop aspects of this theory further,
focusing especially on coherent modules
and finite dimensional modules. We also
discuss the dual Hopf algebra of $P$-algebra
and its comodules. One of our aims is provide
a collection of techniques for calculating
cohomology groups over $P$-algebras and
their duals, in particular giving vanishing
results. Much of our work is implicit in
that of Margolis and others but we are
unaware of systematic discussions in
the literature. We give some examples
illustrating topological applications
which follow easily from our results.
\end{abstract}

\maketitle
\tableofcontents

\section*{Introduction}

This paper is a rewrite of the algebraic
part of~\cite{AB:MSpBousfieldclass} and we
intend writing a companion paper on the
topological applications. Shortly after
posting the earlier version, we discovered
some significant gaps in the calculations
and this led to some substantial reworking
of the algebra which we feel is of enough
interest to present it independently. Indeed,
as far as we are aware, although there has
been significant use of properties of
$P$-algebras in work involving the Steenrod
algebra, there does not appear to be much
discussion of the dual notion so it seems
worthwhile recording some basic results
and in particular the existence of certain
spectral sequences for their cohomology.

The theory of \Palgebras{} of
Margolis~\cite{HRM:Book} which builds on work
of Moore \& Peterson~\cite{JCM-FPP:NearlyFrobAlgs},
provides an important theoretical framework
for understanding modules over the Steenrod
algebra at a prime with many applications in
Algebraic Topology. In this paper we review
some basic results on \Palgebras{} and their
modules, particularly emphasising properties
of coherent modules and their relationship
with finite dimensional modules; we also
mention some properties of pseudo-coherent
modules. Next we consider the dual Hopf algebra
of a \Palgebra{} and its comodules. In order
to set up some Cartan-Eilenberg spectral sequences
we discuss the homological algebra required to
make a bivariant derived functor of the homomorphism
for comodules with the aim of obtaining three
such spectral sequences, one of which involves
dualising to modules over the~\Palgebra{} itself
because comodules over Hopf algebras do not always
have resolutions by projective objects (although
recent work of Salch related to~\cite{AS:ProjComods} 
sheds light on this for connective graded comodules
over a connective graded Hopf algebra). After
reviewing some properties of the mod~$2$ Steenrod
algebra and its dual, we discuss doubling and
then describe some families of subHopf algebras
and their duals. Finally we give some applications
of our results on coherent modules to show vanishing
of homotopy groups of maps between some spectra,
rederiving and extending results of Lin, Margolis
and Ravenel.

An important motivation for our work is to rework
some the algebraic machinery used in Ravenel's
seminal paper~\cite{DCR:Localn} and we will
apply this in the planned sequel. In~\cite{AB:LocFrobAlg}
we have also developed an ungraded analogue of 
the theory of $P$-algebras which sheds light 
on some aspects of it.

\section{\Palgebras{} and their modules}
\label{sec:P-algebras}
We will make use material on \Palgebras{} and
their modules contained in~\cite{HRM:Book}*{chapter~13}.
Here is a summary of our assumptions, conventions
and notations, some of which differ slightly
from those of Margolis.
\begin{itemize}
\item
We will often suppress explicit mention of
internal grading in cohomology and just
write $M$ for $M^*$ when discussing a module
over $A=A^*$. When working in homology we will
usually write $M_*$ or $A_*$.
\item
We will work with graded vector spaces over
a field~$\k$ and in particular, $\k$-algebras
and their (co)modules. In our topological
applications,~$\k=\F_2$ although similar results
for odd primes are easily found.

For a finite type graded vector space $V_*$ we think
of~$V_n$ as dual to $V^n=\Hom_\k(V_n,\k)$, so $V^*$
is the cohomologically graded degree-wise dual, and
bounded below means that~$V^*$ is bounded below. We
can also start with a finite type graded vector space
$V^*$ and form $V_*$ where $V_n=\Hom_\k(V^n,\k)$;
the double dual of $V_*$ is canonically isomorphic
to $V_*$, and vice versa. We will denote the positive
degree part of a graded vector space by~$V^+$; of
course for a graded $\k$-algebra~$R$, $R^+$ is an
ideal which is maximal if~$R$ is also connected,
making~$R$ local.

A graded vector space which is finite dimensional
will be referred to as a \emph{finite}; when~$\k$
is finite this terminology agrees with the use of
finite for a module over a \Palgebra{} by Margolis.
\item
In this paper, a \emph{\Palgebra}~$A$ will always
be a \emph{$P$-Hopf algebra}, i.e., a strictly
increasing union of connected finite dimensional
cocommutative Hopf algebras $A(n)\subset A(n+1)\subset A$.
Thus each~$A(n)$ is a Poincar\'e duality algebra
and we will denote its highest non-trivial degree
by~$\pd(n)$ as it is also the Poincar\'e duality
degree; this number satisfies the inequality
$\pd(n)<\pd(n+1)$. We also stress that each~$A(n)$
is a local graded ring and $A(n)^{\pd(n)}=A(n)^0=\k$.
Other important properties are that~$A$ is free
as a left or right $A(n)$-module and~$A$ is a
coherent $\k$-algebra. Although in general
Margolis does not require \Palgebras{} to be
of finite type, all the examples we consider
have that property and it is required in some
of our homological results so we will assume
it holds.

It is important to note that $P$-algebras are
\emph{coherent} (but not Noetherian) and their
coherent modules play an important r\^ole in
this paper. For general properties of coherence
and pseudo-coherence we refer the reader to
Bourbaki~\cite{Bourbaki:HomAlg}*{ex.~\S3.10},
and also Cohen~\cite{JMC:Coherent} for an
account aimed at an algebraic topology
audience.

In Algebraic Topology the important examples
of \Palgebras{} are the Steenrod algebra~$\StA$
at a prime~$p$ as well as infinite dimensional
sub and quotient Hopf algebras.
\end{itemize}

We will use the following basic result
stated on page~195 of~\cite{HRM:Book}
but left as an exercise.
\begin{prop}\label{prop:Palg-finite}
Suppose that $A$ is a P-algebra. Let~$M$
be a finite $A$-module and~$F$ a bounded
below free $A$-module. Then
\[
\Ext_A^*(M,F)=0.
\]
%\emph{(b)} If $N$ is a finitely presented
%$A$-module then
%\[
%\Hom_A(M,N)=0.
%\]
%Hence $N$ cannot be a colimit of finite
%$A$-modules.
\end{prop}
\begin{proof}
%(a)
By~\cite{HRM:Book}*{theorem~13.12}, bounded
below projective $A$-modules are also injective,
so $\Ext_A^s(M,F)=0$ when $s>0$, therefore
we only have to show that $\Hom_A(M,F)=0$.
It suffices to prove this for the case~$F=A$.

Suppose that $0\neq\theta\in\Hom_A(M,A)$.
The image of $\theta$ contains a simple
submodule in its top degree, so let
$\theta(x)\neq0$ be in this submodule;
then $a\theta(x)=0$ for every positive
degree element $a\in A$. Now for some~$n$,
$\theta(x)\in A(n)^k$ where $k<\pd(n)$,
and by Poincar\'e duality for~$A(n)$
there exists $z\in A(n)$ for which
$0\neq z\theta(x)\in A(n)^{\pd(n)}$.
This gives a contradiction, hence no
such~$\theta$ can exist.
%\\
%(b) Let
%\[
%0\to K\to A^{\oplus k}\to N\to 0
%\]
%be a finite presentation. On applying $\Hom_A(M,-)$
%we obtain an exact sequence
%\[
%0\to \Hom_A(M,K)\to \Hom_A(M,A^{\oplus k})\to \Hom_A(M,N)
%\]

For a different proof, see
Lorentz~\cite{ML:TourRepThy}*{proposition~10.6},
which shows that an infinite dimensional Hopf
algebra has no non-trivial finite dimensional
left or right ideals.
\end{proof}

A particular concern for us will be the
situation where we have a pair of \Palgebras{}
$B\subseteq A$ with~$B$ a subHopf algebra
of~$A$.
\begin{prop}\label{prop:ExtA/B}
For a pair of P-algebras $B\subseteq A$,
\[
\Ext^*_A(A\otimes_B\k,A) \iso \Ext^*_B(\k,A) = \Hom_B(\k,A) = 0.
\]
\end{prop}
\begin{proof}
First we recall a classic result of Milnor \&
Moore~\cite{M&M:HopfAlg}*{proposition~4.9}:
$A$ is free as a left or right $B$-module. This
guarantees the change of rings isomorphism (which
is valid for any subalgebra where~$A$ is a
flat~$B$-module) and the second isomorphism
follows from Proposition~\ref{prop:Palg-finite}.
\end{proof}

Since a \Palgebra{} $A$ is coherent, its finitely
presented modules are its coherent modules, and
they form a full abelian subcategory $\Mod^{\mathrm{coh}}_A$
of~$\Mod_A$ with finite limits and colimits (in
particular it has kernels, cokernels and images).
A coherent $A$-module~$M$ admits a finite presentation
\[
\xymatrix{
A^{\oplus k}\ar[r]^{\pi}& A^{\oplus \ell}\ar[r]& M\ar[r]&0
}
\]
which can be defined over some $A(n)$, i.e.,
there is a finite presentation
\[
\xymatrix{
A(n)^{\oplus k}\ar[r]^{\pi'}&A(n)^{\oplus \ell}\ar[r]& M'\ar[r]&0
}
\]
of $A(n)$-modules and a commutative diagram
of $A$-modules
\[
\xymatrix{
A\otimes_{A(n)}A(n)^{\oplus k}\ar[r]^{\Id\otimes\pi'}\ar@{<->}[d]^\iso
& A\otimes_{A(n)}A(n)^{\oplus \ell}\ar[r]\ar@{<->}[d]^\iso
& A\otimes_{A(n)}M'\ar[r]\ar@{<->}[d]^\iso &0 \\
A^{\oplus k}\ar[r]^{\pi}& A^{\oplus \ell}\ar[r]& M\ar[r]&0
}
\]
with exact rows. It is standard that every
coherent $A$-module admits a resolution by
finitely generated free modules. It is also
true that a homomorphism between $M\to N$
coherent modules is induced from a homomorphism
between finitely generated modules over some
$A(m)$.

For a \Palgebra{} we also have injective
resolutions by finitely generated free
modules.
\begin{prop}\label{prop:Coherent-InjResn}
Let $M$ be a coherent module over a \Palgebra{}
$A$. Then $M$ admits an injective resolution
by finitely generated free modules.
\end{prop}
\begin{proof}
By~\cite{HRM:Book}*{theorem~13.12}, bounded
below projective $A$-modules are injective.

For some $n$, $M\iso A\otimes_{A(n)}M'$
where~$M'$ is a finitely generated $A(n)$-module.
Since~$A(n)$ is a Poincar\'e duality algebra,
it is standard that~$M'$ admits a monomorphism
$M'\to J'$ into a finitely generated free
$A(n)$-module which is also injective
(this is obvious for a simple module and
can be proved by induction on dimension).
By flatness the composition
\[
\xymatrix{
M\ar[r]_(.3){\iso}\ar@/^15pt/[rr] & A\otimes_{A(n)}M'\ar[r]& A\otimes_{A(n)}J'
}
\]
is a monomorphism of $A$-modules into a
finitely generated free module with coherent
cokernel. Iterating this we can build a resolution
of the stated form and since~$A$ is self-injective
this is an injective resolution.
\end{proof}

Our next result summarises the properties of
coherent modules.
\begin{prop}\label{prop:Coherent-AbCat}
The coherent modules over a \Palgebra{} $A$ form
a full subcategory\/ $\Mod^{\mathrm{coh}}_A$ of\/
$\Mod_A$ with enough projectives and injectives,
and all finite limits and colimits.
\end{prop}

We can generalise Proposition~\ref{prop:Palg-finite}.
\begin{prop}\label{prop:Finite->coherent}
Let $M$ be a finite $A$-module and $N$ a coherent
$A$-module. Then
\[
\Ext_A^*(M,N) = 0.
\]
\end{prop}
\begin{proof}
The left exact functor $\Hom_A(M,-)$ has the right
derived functors $\Ext_A^*(M,-)$. By
Proposition~\ref{prop:Palg-finite}, for each finitely
generated free module~$F$, $\Ext_A^*(M,F) = 0$. This
means that~$F$ is $\Hom_A(M,-)$-acyclic, and it is
well-known that these derived functors can be computed
using resolutions by such modules; see~\cite{CAW:HomAlg}
on $F$-acyclic objects and dimension shifting.

Now every coherent $A$-module $N$ admits a resolution
$0\to N\to J^*$ where each $J^s$ is a finitely generated
free module and these are $\Hom_A(M,-)$-acyclic. Therefore
$\Ext_A^*(M,N) = 0$.
%
%Given a short exact sequence of coherent $A$-modules
%\[
%0\to K \to F \to L \to 0
%\]
%where $F$ is a finitely generated free module,
%$\Ext_A^*(M,F) = 0$ by Proposition~\ref{prop:Palg-finite}.
%It follows that for each~$s\geq0$, there is a
%dimension-shifting isomorphism
%\[
%\Ext_A^{s+1}(M,K)\iso\Ext_A^s(M,L)
%\]
%and $\Ext_A^0(M,K)=0$.
%
%We will use induction on $s\geq0$ to prove that
%for every coherent coherent $A$-module~$K$,
%$\Ext_A^s(M,K)=0$. The initial case~$s=0$ holds
%true because by Proposition~\ref{prop:Coherent-InjResn},
%every coherent $A$-module~$K$ embeds into a
%finitely generated free module~$F$, giving
%a short exact sequence as above.
%
%If the result is true for all $s\leq n$,
%then dimension shifting gives
%\[
%\Ext_A^{n+1}(M,K)\iso\Ext_A^n(M,L)=0.
%\]
%So it also holds for $s=n+1$.
\end{proof}

For example, every left $A$-module of the form
\[
A/\!/A(n) = A\otimes_{A(n)}\k \iso A/AA(n)^+
\]
is coherent so it admits such an injective
resolution and then $\Ext_A^s(\k,A/\!/A(n)) = 0$.

Here is a more general statement.
\begin{prop}\label{prop:IndFiniteB->coherent}
Let $A$ be a $P$-algebra and suppose that
$B\subseteq A$ is a subalgebra which is also
a $P$-algebra and that~$A$ is a free left
$B$-module. Let~$L$ be a finite $B$-module
and~$N$ a coherent $A$-module. Then
\[
\Ext^*_A(A\otimes_B L,N) = 0.
\]
\end{prop}
\begin{proof}
We have
\[
\Ext^*_A(A\otimes_B L,N) \iso \Ext^*_B(L,N).
\]
By Proposition~\ref{prop:Coherent-InjResn},
$N$ admits an injection resolution
\[
0\to N\to J^*
\]
by finitely generated free $A$-modules.
For each $s\geq0$,
\[
\Hom_A(A\otimes_B L,J^s)
\iso \Hom_B(L,J^s)
= 0
\]
since a $B$-module homomorphism $L\to J^s$
must factor through finitely many $B$-summands
and therefore it is trivial. Therefore the
cohomology $\Ext^*_B(L,N)$ is trivial. A
similar argument also applies to the case
where~$N$ is a coproduct of coherent
$A$-modules.
\end{proof}

\subsection*{Pseudo-coherent modules over
a coherent ring}\label{sec:pcModules}

In Bourbaki~\cite{Bourbaki:HomAlg}*{ex.~\S3.10},
as well as coherent modules, pseudo-coherent
modules are considered, where a module is
\emph{pseudo-coherent} if every finitely
generated submodule is finitely presented
(so over a coherent ring it is coherent).
The reader is warned that pseudo-coherent
is sometimes used in a different sense; for
example, over a coherent ring the definition
of Weibel~\cite{CAW:Ktheory}*{example~II.7.1.4}
corresponds to our coherent.

Examples of pseudo-coherent modules over
a coherent ring~$A$ include coproducts
of coherent modules such as the finitely
related modules defined by
Lam~\cite{TYL:LectModules&Rings}*{chapter~2\S4}.
However pseudo-coherent modules do not form
a full abelian subcategory of~$\Mod_A$ since
cokernels of homomorphisms between pseudo-coherent
modules need not be pseudo-coherent. Nevertheless,
pseudo-coherent modules do occur quite commonly
when working with coherent modules over $P$-algebras,
for example as tensor products over a Hopf
algebra.

%\begin{defn}\label{defn:psmorphism}
%A homomorphism $\theta\:M\to N$ between
%two pseudo-coherent modules is a
%\emph{pseudo-coherent morphism} if~$\coker\theta$
%is pseudo-coherent. The set of all such
%pseudo-coherent morphisms is denoted
%$\Mod^{\mathrm{p.c.}}_A(M,N)$.
%\end{defn}
%\begin{prop}\label{prop:Modpc}
%There is an abelian subcategory\/~$\Mod^{\mathrm{p.c.}}_A$
%of\/~$\Mod_A$ whose objects are the pseudo-coherent
%modules and whose morphisms are the pseudo-coherent
%morphisms.
%\end{prop}
%\begin{proof}
%We need to check some basic properties. We will
%make repeated use of results in
%Bourbaki~\cite{Bourbaki:HomAlg}*{ex.~\S3.10}.
%
%First notice that for any homomorphism $\theta\:M\to N$
%between pseudo-coherent modules, $\ker\theta$ and
%$\im\theta$ are pseudo-coherent since pseudo-coherence
%is inherited by submodules.
%
%Now suppose that $\phi\:L\to M$ and $\theta\:M\to N$
%are pseudo-coherent morphisms. We want to show
%that $\theta\circ\phi\:L\to N$ is pseudo-coherent.
%There is an induced monomorphism
%\[
%\bar\theta\:M/(\im\phi+\ker\theta)\to N/\im\theta\circ\phi;
%\quad
%x+(\im\phi+\ker\theta) \mapsto \theta(x)+\im\theta\circ\phi,
%\]
%and there is a short exact sequence
%\[
%\xymatrix{
%0\ar[r] &
%\im\theta/\im\theta\circ\phi\ar[r] & N/\im\theta\circ\phi\ar[r] & N/\im\theta\ar[r] & 0
%}
%\]
%where $N/\im\theta$ is pseudo-coherent. Now
%$\im\theta/\im\theta\circ\phi=\im\bar\theta$
%\end{proof}

\begin{lem}\label{lem:pc-colimit}
Let $A$ be a ring and $M$ a pseudo-coherent
$A$-module. Then~$M$ is the union of its
coherent submodules and therefore their
colimit.
\end{lem}
\begin{proof}
Every element~$m\in M$ generates a cyclic
submodule which is coherent.
\end{proof}

\begin{lem}\label{lem:pc-extended}
Let $A$ be a coherent ring and $B$ an Artinian
subring where~$A$ is flat as a right $B$-module.
Then every extended $A$-module $A\otimes_BN$
is pseudo-coherent.
\end{lem}
\begin{proof}
Let $U\subseteq A\otimes_BN$ be a finitely
generated submodule. Taking a finite generating
set and expressing each element as a sum of basic
tensors we find that $U\subseteq A\otimes_BU'$
where $U'\subseteq N$ is a finitely generated
submodule. As~$B$ is Artinian and so Noetherian,~$U'$
is finitely presented, so by flatness of~$A$,
$A\otimes_BU'$ is a finitely presented $A$-module.
Since $A$ is coherent,~$U$ is also finitely
presented.
\end{proof}

An important special case of this occurs when~$A$
is a coherent algebra over a field~$\k$ and~$B$
is a finite dimensional subalgebra. If~$N$ is
a $B$-module then it is locally finite, i.e.,
every element is contained in a finite
dimensional submodule.

%\section{Tensor products}\label{sec:TensorProds}
In general, the tensor product of two coherent
modules over a Hopf algebra that is a $P$-algebra
is not a coherent module. However, as we will
soon see, it turns out that it is pseudo-coherent.

\begin{assump}\label{assump:TensorProds}
We will assume from now on that~$A$ is a cocommutative
Hopf algebra over a field~$\k$, and $B\subseteq A$
a subHopf algebra where~$A$ is $B$-flat as a left
and right $B$-module.
\end{assump}

Now given two left $A$-modules $L$ and~$M$,
their tensor product~$L\otimes M$ is a left
$A$-module with the diagonal action given by
\[
a\cdot(\ell\otimes m)
= \sum_i a'_i\ell\otimes a''_im,
\]
where the coproduct on~$a$ is
\[
\psi a = \sum_i a'_i\otimes a''_i.
\]

In particular, given a left $A$-module~$M$ and
a left $B$-module~$N$, the tensor product of~$M$
and $A\otimes_B N$ is a left $A$-module
$M\otimes(A\otimes_B N)$. There is an isomorphism
of $A$-modules
\begin{equation}\label{eq:TwistingIso}
M\otimes(A\otimes_B N)
\xrightarrow[\iso]{\Theta}
A\otimes_B(M\otimes N);
\quad
m\otimes(a\otimes n)\mapsto
\sum_i a'_i\otimes(\bar{a''_i}m\otimes n)
\end{equation}
where $\bar{x}=\chi(x)$ and $M\otimes N$ is a left
$B$-module with the diagonal action.

A particular instance of this is
\begin{equation}\label{eq:A//B}
(A/\!/B)\otimes(A/\!/B) \iso A\otimes_B(A/\!/B)
\end{equation}
where
\[
A/\!/B=A/AB^+\iso A\otimes_{B}\k.
\]
We will be especially interested in the case
where the $B$-module $A/\!/B$ is a coproduct
of finitely generated modules.

\begin{lem}\label{lem:TensorProd-extended}
Suppose that $A$ is coherent and $B$ is finite
dimensional. Let~$M$ be a left $A$-module and~$N$
a left $B$-module. Then the $A$-module
$M\otimes(A\otimes_BN)$ is pseudo-coherent.
\end{lem}
\begin{proof}
Using \eqref{eq:TwistingIso},
\[
M\otimes(A\otimes_BN) \iso A\otimes_B(M\otimes N)
\]
which is pseudo-coherent by Lemma~\ref{lem:pc-extended}.
\end{proof}

\begin{cor}\label{cor:TensorProd-extended}
Suppose that $A$ is a $P$-algebra and that
$M_1$ and $M_2$ are two coherent $A$-modules.
Then $M_1\otimes M_2$ is pseudo-coherent.
\end{cor}
\begin{proof}
Every coherent $A$-module is induced from
a finitely generated $A(n)$-module for some~$n$.
By choosing a large enough~$n$ we can assume
that $M_1\iso A\otimes_{A(n)}M_1'$ for some
finitely generated $A(n)$-module~$M_1'$. Then
\[
M_1\otimes M_2 \iso A\otimes(M_1'\otimes M_2),
\]
so it is a pseudo-coherent $A$-module.
\end{proof}

Pseudo-coherence itself is probably of limited
practical use, but stronger versions are perhaps
more likely to be of use for computational
purposes.

\begin{defn}\label{defn:spc}{\ }
\begin{itemize}
\item
A module over a coherent ring is \emph{strongly
pseudo-coherent\/} if it is a coproduct of
coherent modules.
\item
A module $M$ over a $P$-algebra~$A$ is
\emph{$n$-strongly pseudo-coherent} if
\[
M\iso A\otimes_{A(n)}M'
\]
where the $A(n)$-module~$M'$ is a coproduct
of finitely generated $A(n)$-modules.
\end{itemize}
\end{defn}

Over a coherent ring every free module is strongly
pseudo-coherent as is every finitely related
module since it is the sum of a coherent module
and a free module.

Of course an $n$-strongly pseudo-coherent module
over a $P$-algebra~$A$ is strongly pseudo-coherent.
For a given~$n$, the cyclic $A$-module
$A/\!/A(n)=A/AA(n)^+\iso A\otimes_{A(n)}\k$ can
be viewed as an $A(n)$-module, and by~\eqref{eq:A//B},
\[
A/\!/A(n)\otimes A/\!/A(n) \iso A\otimes_{A(n)}A/\!/A(n).
\]
If $A/\!/A(n)$ is a coproduct of finitely generated
(i.e., finite dimensional) $A(n)$-modules then
$A/\!/A(n)\otimes A/\!/A(n)$ is $n$-strongly
pseudo-coherent. This occurs when~$A$ is the Steenrod
algebra for a prime~$p$.
\begin{examp}\label{examp:StA-BGitsplitting}
When $A=\StA$ is the mod~$2$ Steenrod algebra it
is known that there is a coproduct decomposition
of the $\StA(n)$-module $\StA/\!/\StA(n)$ into
finite dimensional modules which are generalised
Brown-Gitler modules. For details see
Behrens et al~\cite{MB-MAH-MJH-MM:v2to32}*{corollary~5.6}.
\end{examp}

\section{The dual of a \Palgebra{} and its comodules}
\label{sec:DualP-algebras}

The theory of \Palgebras{} can be dualised: given
a finite type \Palgebra~$A$, we define its (graded)
dual~$A_*$ by setting $A_n=\Hom_\k(A^n,\k)$ and
making this a commutative Hopf algebra by dualising
the structure maps of~$A$ (doing this the finite
type condition is essential). We will refer to
the dual Hopf algebra of a (finite type)
\Palgebra{} as a \emph{$P_*$-algebra}; although
this is not standard terminology it seems appropriate 
and convenient.

When working with comodules over a $P_*$-algebra~$A_*$
we will use homological grading. For left
$A_*$-comodules which are bounded below and
of finite type there is no significant difference
between working with them or their (degree-wise)
duals as $A$-modules. In particular,
\[
 \Cohom_{A_*}(M_*,N_*) \iso \Hom_{A}(N^*,M^*),
\]
where $M^n=\Hom_{\F_2}(M_n,\F_2)$ and $M^*$ is
made into a left $A$-module using the antipode.
More generally,
\[
\Coext^{s,*}_{A_*}(M_*,N_*) \iso \Ext^{s,*}_{A}(N^*,M^*),
\]
where $\Coext^{s,*}_{A_*}(M_*,-)$ denotes the
right derived functor of $\Cohom_{A_*}(M_*,-)$,
which can be computed using extended comodules
which are injective comodules here since we
are working over a field.

Here are the dual versions of
Propositions~\ref{prop:Palg-finite}
and~\ref{prop:ExtA/B}. Recall that a \emph{cofree}
or \emph{extended} $A_*$-comodule is one of the
form $A_*\otimes W_*$ where~$W_*$ is a graded
vector space.
\begin{prop}\label{prop:P*alg-finite}
Suppose that $A_*$ is a $P_*$-algebra. Let~$L_*$
a bounded below cofree $A_*$-comodule and let~$M_*$
be a finite $A_*$-comodule. Then
\[
\Coext_{A_*}^*(L_*,M_*)=0.
\]
\end{prop}
\begin{prop}\label{prop:CoextA/B}
For a surjective morphism of $P_*$-algebras
$A_*\to B_*$,
\[
\Coext^{*,*}_{A_*}(A_*,A_*\square_{B_*}\k)
\iso
\Coext^{*,*}_{B_*}(A_*,\k) = \Cohom_{B_*}(A_*,\k) = 0.
\]
\end{prop}
\begin{defn}\label{defn:coherentcomodule}
If $A_*$ is a $P_*$-algebra then an
$A_*$-comodule~$M_*$ is \emph{coherent}
if its dual~$M^*$ is a coherent $A$-module;
this is equivalent to the existence of an
exact sequence of $A_*$-comodules
\[
0\to M_*\to A_*\otimes U_*\to A_*\otimes V_*
\]
where $U_*,V_*$ are finite $\k$-vector spaces.
\end{defn}

Here is a dual version of Proposition~\ref{prop:Coherent-InjResn}.
\begin{prop}\label{prop:CoherentComod-ProjResn}
Let $M_*$ be a coherent comodule over a $P_*$-algebra{}
$A_*$. Then~$M_*$ admits a projective resolution
by finitely generated cofree comodules.
\end{prop}
\begin{proof}
Take an injective resolution of $M^*$ as
in Proposition~\ref{prop:Coherent-InjResn}
and then take duals to obtain a projective
resolution.
\end{proof}

Notice that since $A_*$ is an injective comodule
we have
\[
\Coext_{A_*}^*(A_*,A_*)
\iso \Hom_\k(A_*,\k)
\iso A.
\]

\begin{prop}\label{prop:CoherentComod-finite}
Let $M_*$ be a coherent comodule over a $P_*$-algebra{}
$A_*$ and let $N_*$ be a finite $P_*$-comodule.
Then
\[
\Coext_{A_*}^*(M_*,N_*)=0.
\]
\end{prop}
\begin{proof}
Let $P_{\bullet,*}\to M_*\to 0$ be a resolution
of~$M_*$ by cofree comodules. Then by
Proposition~\ref{prop:P*alg-finite}, for each
$s\geq0$ we have
\[
\Cohom_{A_*}(P_{s,*},N_*) = 0,
\]
and the result follows.
\end{proof}

\begin{rem}\label{rem:comodules-modules}
For bounded below finite type comodules over
a $P_*$-algebra $A_*$ dual to a \Palgebra{}
$A$, taking degree-wise duals defines an
equivalence of categories
\[
\xymatrix{
\Comod_{A_*}^{\flat,\,\mathrm{f.t.}}\ar@<.5ex>[rr]^(.5){(-)^*}
&& (\Mod_{A}^{\flat,\,\mathrm{f.t.}}\ar@<.5ex>[ll]^(.5){(-)_*})^\op
}
\]
between the $A_*$-comodule and the $A$-module
categories. Moreover, these functors are exact,
so this equivalence identifies injective
comodules (which are retracts of extended
comodules) with bounded below projective
modules. By~\cite{HRM:Book}*{theorem~13.12},
bounded below projective $A$-modules are
injective so it also identifies bounded
below injective comodules as projective
objects (this is not true in general of
course).  In fact this equivalence fits
into a bigger diagram
\begin{equation}\label{eq:Comod->Mod}
\xymatrix{
&& \Mod_{A}^{\sharp,\,\mathrm{f.t.}}\ar@<.5ex>[d]^(.5){(-)^*} \\
\Comod_{A_*}^{\flat,\,\mathrm{f.t.}}
\ar@<.5ex>[rr]^(.5){(-)^*}
\ar@/^15pt/[urr] \ar@/_15pt/[drr]
&& (\Mod_{A}^{\flat,\,\mathrm{f.t.}})^\op
\ar@<.5ex>[ll]^(.5){(-)_*}\ar@<.5ex>[u]^(.5){(-)_*}\ar[d]  \\
&& \MOD_{A}^\op
}
\end{equation}
where $\Mod_{A}^{\natural,\,\mathrm{f.t.}}$
denotes the category of finite type bounded
below homologically graded $A$-modules
(with $A$ acting by decreasing degree),
$\Mod_{A}^{\flat,\,\mathrm{f.t.}}$ denotes
the category of finite type bounded below
cohomologically graded $A$-modules
and $\MOD_{A}$ denoting the category
of all $A$-modules. All of the functors
here are exact.
\begin{rem}\label{rem:locfinmods}
Each object $M_*$ of
$\Mod_{A}^{\natural,\,\mathrm{f.t.}}$
is a locally finite $A$-module, i.e.,
it is a union of finite modules. It is
convenient to regrade~$M_*$ so that~$M_n$
is given cohomological degree~$-n$ and
then multiplication by a positive
degree element of~$A$ increases this
cohomological degree. Then by
Proposition~\ref{prop:Finite->coherent}
we have for any coherent $A$-module~$N$,
$\Ext_{A}^*(M_*,N) = 0$.
\end{rem}

For a fixed $A_*$-comodule $M_*$, the
functor
\[
\Cohom_{A_*}(M_*,-)=\Comod_{A_*}^{\flat,\,\mathrm{f.t.}}(M_*,-)
\to\Mod_\k^{\flat,\,\mathrm{f.t.}}
\]
is left exact and has right derived functors
$\Coext^*_{A_*}(M_*,-)$. Since
\[
\Cohom_{A_*}(M_*,-)\iso\Hom_{A}((-)^*,M^*)
=\Mod_{A}^{\flat,\,\mathrm{f.t.}}((-)^*,M^*)
=\MOD_{A}((-)^*,M^*)
\]
and injective comodules are sent to projective
modules, we also have
\begin{equation}\label{eq:Coext-Ext}
\Coext^*_{A_*}(M_*,-)\iso \Ext^*_{A}((-)^*,M^*).
\end{equation}

The contravariant functor
$\Comod_{A_*}^{\flat,\,\mathrm{f.t.}}\to\MOD_{A}^\op$
allows us to define cohomological invariants
of comodules using injective resolutions
in $\MOD_{A}$ as a substitute for projective
resolutions in
$\Comod_{A_*}^{\flat,\,\mathrm{f.t.}}$.
In effect for a comodule~$N_*$ we define
\[
\Coext^*_{A_*}(-,N_*)=\Ext_A(N^*,(-)^*).
\]
Of course this is calculated using injective
resolutions of $A$-modules; since $\Ext_A(-,-)$
is a balanced functor,~\eqref{eq:Coext-Ext}
implies that $\Coext^*_{A_*}(-,-)$ is too,
whenever we can use projective comodule
resolutions in the first variable. For
example, if we restrict to the subcategory
of coherent comodules we obtain balanced
bifunctors
\[
\Coext^s_{A_*}(-,-)\:
(\Comod_{A_*}^{\mathrm{coh}}(-,-))^{\op}
\otimes\Comod_{A_*}^{\mathrm{coh}}(-,-)
\to \Mod_\k^{\flat,\,\mathrm{f.t.}}.
\]
\end{rem}

Given a surjection of $P_*$-algebras
$A_*\to B_*$ there are adjunction
isomorphisms of the form
\begin{align}\label{eq:Cohom-Cotor-1}
\Cohom_{A_*}(-,-)
&\iso
\Cohom_{B_*\backslash\!\backslash A_*}
((B_*\backslash\!\backslash A_*)\square_{A_*}(-),-),
 \\ \label{eq:Cohom-Cotor-2}
\Cohom_{B_*}(-,-)
&\iso
\Cohom_{A_*}(-,A_*\square_{B_*}(-)),
\end{align}
where
$B_*\backslash\!\backslash A_* = \k\square_{B_*}A$.
Later we will use these adjunctions to
construct composite functor spectral
sequences.

\begin{rem}\label{rem:ASalch}
Since writing  this we became aware of the
revised version of Salch~\cite{AS:ProjComods},
where it is shown that the category of graded
connected comodules over a graded connected
Hopf algebra over a field has enough projectives.
As our results are stronger but more restricted,
we feel it worthwhile presenting them despite
the greater generality of Salch's result.
\end{rem}

%Finally we record a technical result that will
%be used in calculations.
%\begin{lem}\label{lem:colim-Ext}
%Let $A$ be a \Palgebra, $M$ a finite $A$-module
%and $\ds N=\colim_iN_i$ be a filtered colimit
%of $A$-modules. Then
%\[
%\Ext^*_A(M,N) \iso \colim_i\Ext^*_A(M,N_i).
%\]
%\end{lem}
%\begin{proof}
%Since $M$ is finite dimensional we can build
%a free resolution $P_*\to M$ where each free
%$A$-module
%\end{proof}

\section{Some homological algebra}
\label{sec:HomAlg}

In this section we describe some Cartan-Eilenberg
spectral sequences for comodules over a commutative
Hopf algebra over a field. Some of these are
similar to other examples in the literature
such as that for computing~$\Cotor$ for Hopf
algebroids in~\cite{DCR:GreenBook}.

To ease notation, in this section we suppress
internal gradings and assume that all our
objects are connective and of finite type
over a field~$\k$. We refer to the classic
\cite{M&M:HopfAlg} as well as the more
recent \cite{JPM&KP:MoreConcise} for notation
and basic ideas about graded Hopf algebras.

Before discussing cohomology for comodules,
we will recall the dual theory to modules
over algebras, where there are classical
Cartan-Eilenberg spectral sequences of~\cite{HC&SE:HomAlg}
for a normal sequence of Hopf algebras
over a field~$\k$,
\begin{equation}\label{eq:CESS-algebraseq}
R\to S \to S/\!/R
\end{equation}
where $S$ is a free $R$-module. Then for a left
$S/\!/R$-module~$L$ and a left $S$-module~$M$
there is a spectral sequence of the form
\begin{equation}\label{eq:CESS-algebras-1}
\mathrm{E}_2^{s,t} =
\Ext^s_{S/\!/R}(L,\Ext^t_{R}(\k,M))
\Lra \Ext^{s+t}_{S}(L,M).
\end{equation}
There is another similar spectral sequence
for a left $S$-module~$M$ and a left
$S/\!/R$-module~$N$ which has the form
\begin{equation}\label{eq:CESS-algebras-2}
\mathrm{E}_2^{s,t} =
\Ext^s_{S/\!/R}(\Tor^t_{R}(\k,M),N)
\Lra \Ext^{s+t}_{S}(M,N).
\end{equation}

Since $\Ext$ and $\Tor$ are balanced functors,
one approach to setting these spectral sequences
is by resolving both variables and using double
complex spectral sequences. However, they can
instead be viewed as composite functor spectral
sequences obtained using injective or projective
resolutions of the $S$-module~$M$. When the
algebras and modules are graded $\k$-vector spaces,
these spectral sequences are tri-graded; also,
in topological applications,~\eqref{eq:CESS-algebraseq}
is often a sequence of cocommutative Hopf algebras.

Now suppose we have a sequence of homomorphisms
of connected commutative graded Hopf algebras
over~$\k$,
\[
K\backslash\!\backslash H \rightarrowtail H \twoheadrightarrow K,
\]
where in the notation of \cite{M&M:HopfAlg}*{definition~3.5},
\[
K\backslash\!\backslash H
= \k\square_K H
= H \square_K\k
\subseteq H .
\]
We also assume given a left
$K\backslash\!\backslash H$-comodule
$M$ and a left $H $-comodule~$N$. Of course
$M$ and $N$ inherit structures of $H $-comodule
and $K$-comodule respectively, where~$M$
is trivial as a $K$-comodule. Our aim is
to calculate $\Coext^*_H (M,N)$, the right
derived functor of
\[
\Cohom_H (M,-)\:
\Comod_{K\backslash\!\backslash H}\to\Mod_\k;
\quad N\mapsto \Cohom_H (M,N).
\]

Following Hovey~\cite{MH:HtpyThyComods}, we
will write $U\overset{H}\wedge V$ to indicate
the tensor product of two $H $-comodules
$U\otimes V=U\otimes_\k V$ with the diagonal
coaction given by the composition

\[
\xymatrix{
U\otimes V\ar[r]_(.3){\mu\otimes\mu}\ar@/^18pt/[rrr]
& (H \otimes U)\otimes (H \otimes V)\ar[r]_{\iso}
& (H \otimes H)\otimes (U\otimes V)\ar[r]_(.55){\phi\otimes\Id}
& H \otimes (U\otimes V).
}
\]
For a vector space $W$, the notation $U\otimes W$
will be used to denote $H $-comodule with
coaction

\[
\xymatrix{
U\otimes W\ar[r]_(.4){\mu\otimes\Id}\ar@/^14pt/[rr]
& (H \otimes U)\otimes W\ar[r]_(.5){\iso}
& H \otimes(U\otimes W)
}
\]
carried on the first factor alone.

If $L$ is a left $H $-comodule, then there
is a well-known isomorphism of left
$H$-comodules
\begin{equation}\label{eq:Comodtwisting}
K\backslash\!\backslash H \overset{H }\wedge L
=
(H \square_K\k)\overset{H }\wedge L
\iso
H \square_K L,
\end{equation}
We can also regard
$K\backslash\!\backslash H =\k\square_K H $
as a right $H $-comodule to form the left
$K\backslash\!\backslash H $-comodule
\begin{equation}\label{eq:Comodtwisting-Sigma}
K\backslash\!\backslash H \square_H  L
= (\k\square_K H )\square_H  L
\iso \k\square_K L;
\end{equation}
in particular, if $L$ is a trivial $K$-comodule
then as left $K\backslash\!\backslash H$-comodules,
\begin{equation}\label{eq:Comodtwisting-trivial}
K\backslash\!\backslash H \square_H  L
\iso L.
\end{equation}

We will use two more functors
\[
\Comod_{H }\to\Comod_{K\backslash\!\backslash H};
\quad
N\mapsto
K\backslash\!\backslash H\square_H  N
= (\k\square_K H)\square_H  N
\iso \k\square_K N
\]
and
\[
\Comod_{K\backslash\!\backslash H}\to\Mod_\k;
\quad
N\mapsto  \Cohom_{K\backslash\!\backslash H}(M,N).
\]
Notice that there is a natural isomorphism
\[
 \Cohom_{K\backslash\!\backslash H}(M,K\backslash\!\backslash H\square_H (-))
\iso  \Cohom_H (M,-)
\]
and for an injective $H $-comodule~$J$,
$K\backslash\!\backslash H \square_H  J$
is an injective $K\backslash\!\backslash H$-comodule.
This means we are in a situation where we
have a Grothendieck composite functor spectral
sequence which in this case is a form of
Cartan-Eilenberg spectral sequence; for details
see \cite{CAW:HomAlg}*{section~5.8} for example.

\begin{prop}\label{prop:CESS-1}
Let $M$ be a left\/
$K\backslash\!\backslash H$-comodule and $N$
a left\/ $H $-comodule. Then there is a first
quadrant cohomologically indexed spectral
sequence with
\[
\mathrm{E}_2^{s,t} =
\Coext_{K\backslash\!\backslash H}^s(M,\Cotor_K^t(\k,N))
\Lra
\Coext_H ^{s+t}(M,N).
\]
If $N$ is a trivial\/ $K$-comodule then
\[
\mathrm{E}_2^{s,t} \iso
\Coext_{K\backslash\!\backslash H}^s(M,
\Cotor_K^t(\k,\k)\overset{K\backslash\!\backslash H}\wedge N).
\]
\end{prop}

There is another spectral sequence that
we will use whose construction requires
that one of the Hopf algebras involved
is a $P_*$-algebra. The reason for this
is discussed in Remark~\ref{rem:comodules-modules}:
in the category of finite type connected
comodules, extended comodules are projective
objects.

\begin{prop}\label{prop:CESS-2}
Assume that $H $ and $K\backslash\!\backslash H$
are $P_*$-algebras. Let~$M$ be a left\/
$H$-comodule which admits a projective resolution
and let~$N$ be a bounded below left\/
$K\backslash\!\backslash H$-comodule. Then there
is a first quadrant cohomologically indexed
spectral sequence with
\[
\mathrm{E}_2^{s,t} =
\Coext_{K\backslash\!\backslash H}^s(\Cotor_K^t(\k,M),N)
\Lra
\Coext_H ^{s+t}(M,N).
\]
If $M$ is a trivial\/ $K$-comodule then
\[
\mathrm{E}_2^{s,t} \iso
\Coext_{K\backslash\!\backslash H}^s
(\Cotor_K^t(\k,\k)\overset{K\backslash\!\backslash H}\wedge M,N).
\]
\end{prop}
\begin{proof}
The construction is similar to the other one,
and involves expressing $\Cohom_H (-,N)$
as a composition
\[
\Cohom_{K\backslash\!\backslash H}(-,N)\circ (K\backslash\!\backslash H\square_H (-))
=
\Cohom_{K\backslash\!\backslash H}(K\backslash\!\backslash H\square_H (-),N)
\iso \Cohom_H (-,N).
\]
The functor
$K\backslash\!\backslash H \square_H (-)\:
\Comod_H\to\Comod_{K\backslash\!\backslash H}$
sends projective objects to projective objects
(see Remark~\ref{rem:comodules-modules}), so
the standard construction can be used to give
a spectral sequence.
\end{proof}

Of course the condition that~$M$ admits a projective
resolution is crucial; in the case of $P_*$-algebras
this is satisfied if~$M$ is a coherent comodule.

\section{Finite comodule filtrations}\label{sec:Filtrations}

In this section we will give some results based
on a particular kind of comodule filtration.

We will make use of the adjoint coaction dual
to the adjoint action which is given a thorough
treatment in
Singer~\cite{WMS:SteenrodSqsSS}*{chapter~4};
the dualisation to the comodule setting is
straightforward and details are left to the
reader. For our purposes we need to know the
following: If $B\subseteq A$ is a conormal
subHopf algebra of a commutative Hopf algebra
over a field~$\k$, then there is a (left)
\emph{adjoint coaction} $A\to A\otimes A$
which is the composition of the solid arrows
in the following commutative diagram, where
$\mathrm{T}$ is the switch map.
\begin{equation}\label{eq:AdjCoaction}
\xymatrix{
B \ar@{.>}[d]\ar@/^20pt/@{.>}[ddrrrrrr] &&&&& \\
A\ar[d]_{\psi}\ar@/^20pt/[ddrrrrrr] &&&&&&  \\
A\otimes A\ar[d]|-{\psi\otimes\Id=\Id\otimes\psi} &&&&&& B\otimes B\ar[d] \\
A\otimes A\otimes A\ar[r]
&A\otimes A\otimes A\ar@{<->}[rr]_{\Id\otimes\mathrm{T}}^\iso
&&A\otimes A\otimes A\ar@{<->}[rr]_{\Id\otimes\chi\otimes\Id}^\iso
&&A\otimes A\otimes A\ar[r]_(.58){\phi\otimes\Id}
&A\otimes A \\
}
\end{equation}
By \cite{WMS:SteenrodSqsSS}*{proposition~4.24},
the adjoint coaction makes~$A$ and any conormal
subHopf algebra~$B$ into an $A$-comodule Hopf
algebra; furthermore, for any left $A$-comodule~$M$
(which also becomes a left $A/\!/B$-comodule
through the projection $A\to A/\!/B$), there
is an induced left $A$-coaction on
$\Cotor_{A/\!/B}(\k,M)$ that factors through
a left $B$-coaction.

\begin{equation}\label{eq:Cotor-AdjActionFactors}
\xymatrix{
\Cotor_{A/\!/B}(\k,M)\ar[dr]\ar[rr]
& & A\otimes\Cotor_{A/\!/B}(\k,M)  \\
& B\otimes\Cotor_{A/\!/B}(\k,M)\ar[ur] &
}
\end{equation}

Now let $C$ be a cocommutative coalgebra
over a field~$\k$. A $C$-comodule~$M$ is
\emph{unipotent} if it has a finite length
descending filtration by subcomodules
\[
M=M^\ell\supset M^{\ell-1}\supset\cdots
    \supset M^1\supset M^0=0
\]
where each quotient comodule $M^{i+1}/M^i$
has trivial coaction. We will refer to
such a filtration as a \emph{unipotent
filtration of length}~$\ell$. Every
comodule~$M$ contains a \emph{primitive
sequence} of subcomodules
\[
M\supseteq \cdots \supseteq M^{[i+1]}\supseteq M^{[i]}
\supseteq \cdots \supseteq M^{[2]}\supseteq M^{[1]}
\supseteq M^{[0]} = 0
\]
defined recursively by
\[
M^{[i]} = \pi_{i-1}^{-1}\Prim_C(M/M^{[i-1]})
\]
where $\pi_{i-1}\:M\to M/M^{[i-1]}$ is
the quotient homomorphism, so $\pi_{i-1}$
induces an isomorphism
\[
M^{[i]}/M^{[i-1]}\xrightarrow{\iso}\Prim_C(M/M^{[i-1]}).
\]
In general this need not be exhaustive
or become stable, but if it has both
these properties then it is a unipotent
filtration of~$M$ and we will say that~$M$
has a \emph{finite primitive filtration}.
\begin{lem}\label{lem:unipotent->prim}
Suppose that the comodule~$M$ is unipotent.
Then the primitive sequence of~$M$ is
finite.
\end{lem}
\begin{proof}
Suppose that
\[
M=M^\ell\supset M^{\ell-1}\supset\cdots
    \supset M^1\supset M^0=0
\]
is a unipotent filtration. Then we clearly
have $M^1\subseteq M^{[1]}$. Now suppose
that for some $k\geq1$, $M^k\subseteq M^{[k]}$.
There is a commutative diagram of comodules
\[
\xymatrix{
M^{k+1}\ar@{^{(}->}[r]\ar[d] & M\ar@{=}[r]\ar[d]
& M\ar[d] \\
M^{k+1}/M^k\ar@{^{(}->}[r]\ar@/_15pt/[dr] & M/M^k\ar[r] & M/M^{[k]} \\
& \Prim_C(M/M^k)\ar[r]\ar@{^{(}->}[u]
& M^{[k+1]}/M^{[k]}\ar@{^{(}->}[u]
}
\]
from which it follows that $M^{k+1}\subseteq M^{[k+1]}$.
By Induction we find that $M^{n}\subseteq M^{[n]}$
for all~$n\geq1$ and in particular
$M = M^{\ell}\subseteq M^{[\ell]}$, so $M^{[\ell]}=M$.
\end{proof}

\begin{lem}\label{lem:unipotent-2from3}
Let $L,M,N$ be $C$-comodules fitting
into a short exact sequence
\[
0\to L\to M\to N\to 0.
\]
Then $M$ is unipotent if and only if
$L$ and $N$ are unipotent.
\end{lem}
\begin{proof}
Suppose that $M$ and $N$ are unipotent.
A unipotent filtration of~$N$ pulls back
to a comodule filtration of~$M$ where
each stage contains the image of~$L$,
and this can be extended to a unipotent
filtration of~$M$ using a unipotent
filtration of~$L$.

Suppose that $M$ has a finite primitive
filtration
\[
M=M^{[\ell]}\supset M^{[\ell-1]}\supset\cdots
    \supset M^{[1]}\supset M^{[0]}=0.
\]
Set $L^1 = \Prim_C(L) = L\cap M^{[1]}$.
Then
\[
\Prim_C(L/L^1) \iso \Prim_C(L+M^{[1]}/M^{[1]})
= \bigl((L+M^{[1]})\cap M^{[2]}\bigr)/M^{[1]}.
\]
Now for each $1\leq k\leq\ell$ define
$L^k = L\cap M^{[k]}$. Notice that
$L^k = L\cap M^{[k]}\subseteq =L^{k+1}$
and $L^\ell=L$. Also,
\[
L^{k+1}/L^k \iso (L^{k+1}+ M^{[k]})/M^{[k]}
\subseteq \bigl((L + M^{[k]})\cap M^{[k+1]}\bigr)/M^{[k]}
\subseteq M^{[k+1]}/M^{[k]}
\]
so the coaction of $L^{k+1}/L^k$ is trivial.
Now taking $N=M/L$ define
\[
N^k = (L+M^{[k]})/L \iso M^{[k]}/L^k\cap M^{[k]}
\]
so that
\[
N^{k+1}/N^k \iso (L + M^{[k+1]})/(L + M^{[k]})
\]
which is easily seen to have trivial coaction.
\end{proof}

\begin{rem}\label{rem:Unipotent-induced}
Suppose that $M$ is a unipotent $C$-comodule.
Then $M$ is also unipotent as a $D$-comodule
where $C\to D$ is a surjective morphism of
coalgebras. Also, if the coaction factors
through a $C'$-coaction $M\to C_0\otimes M$
where $C'\subseteq C$ is a subcoalgebra,
then $M$ is a unipotent $C_0$-comodule.
\end{rem}

Our main use of such unipotent filtrations
is to situations described in the next result.
\begin{prop}\label{prop:Unipotent-Cotor}
Let $B\subseteq A$ be a conormal subHopf
algebra of a commutative Hopf algebra over
a field\/~$\k$. Suppose that~$M$ is a
unipotent left $A$-comodule and that for
every~$k\geq0$,\/ $\Cotor^k_{A/\!/B}(\k,\k)$
is a unipotent left $A$-comodule. Then
each\/ $\Cotor^k_{A/\!/B}(\k,M)$ is a
unipotent $A$-comodule.
\end{prop}
\begin{proof}
We remark that the long exact sequence for
$\Cotor^k_{A/\!/B}(\k,-)$ used below is
one of $A$-comodules; the dual module
case follows from the details given in
Singer~\cite{WMS:SteenrodSqsSS}*{chapter~4}.

We prove this by induction on the length
$\ell$ of the unipotent filtration on~$M$.
When $\ell=1$,
\[
\Cotor^k_{A/\!/B}(\k,M) \iso
\Cotor^k_{A/\!/B}(\k,\k)\otimes M
\]
and the result holds. Now suppose that it
holds for $\ell<n$ and let $\ell=n$. Consider
the short exact sequence
\begin{equation}\label{eq:Unipotent-ses}
0\to M^{[n-1]}\to M\to M/M^{[n-1]}\to0
\end{equation}
where we know that for all $k\geq0$,
\[
\Cotor^k_{A/\!/B}(\k,M^{[n-1]}),
\quad
\Cotor^k_{A/\!/B}(\k,M^{[n]}/M^{[n-1]})
\]
are unipotent $A$-comodules. On applying
$\Cotor^k_{A/\!/B}(\k,-)$ to~\eqref{eq:Unipotent-ses}
we obtain a long exact sequence where
in degree~$k$ we have
\[
\xymatrix{
\cdots\ar[r]
& \Cotor^k_{A/\!/B}(\k,M^{[n-1]})\ar[r]
& \Cotor^k_{A/\!/B}(\k,M)\ar[r]
& \Cotor^k_{A/\!/B}(\k,M^{[n-1]})\ar[r]
& \cdots \\
}
\]
so by Lemma~\ref{lem:unipotent-2from3}
there is a short exact sequence
\[
0\to N'_k\to\Cotor^k_{A/\!/B}(\k,M)\to N''_k\to0
\]
for unipotent comodules $N'_k$ and $N''_k$.
Again using Lemma~\ref{lem:unipotent-2from3}
we find that $\Cotor^k_{A/\!/B}(\k,M)$ is
unipotent.
\end{proof}
\begin{rem}\label{rem:Unipotent-Cotor}
Although we have stated the last result
in terms the $A$-comodule structure, in
our applications we will use it by passing
to a conormal quotient Hopf algebra of~$A$.
\end{rem}

%!!!!!!!!!!!!!!!!!!!
%
%We will use a generalisation of these results.
%The proof is essentially the same as before.
%\begin{prop}\label{prop:Unipotent-Cotor-general}
%Let $B\subseteq A$ and $C\subseteq A$ be
%conormal subHopf algebras of a commutative
%Hopf algebra over a field\/~$\k$ with
%$C\subseteq B$. Suppose that~$M$ is a left
%$A$-comodule which is a unipotent as a
%$A/\!/C$-comodule and that for
%every~$k\geq0$,\/ $\Cotor^k_{A/\!/B}(\k,\k)$
%is a unipotent left $A/\!/C$-comodule.
%Then each\/ $\Cotor^k_{A/\!/B}(\k,M)$
%is a unipotent $A/\!/C$-comodule.
%\end{prop}
%
%??????????????????

\begin{prop}\label{prop:Unipotent-Cotor-general}
Let $B\subseteq A$ be a conormal subHopf
algebra of a commutative Hopf algebra over
a field\/~$\k$ where $B$ is cocommutative.
Let $M$ be a left $A$-comodule such that
for every~$k\geq0$,\/ $\Cotor^k_{A/\!/B}(\k,M)$
is a unipotent left $B$-comodule. Then each\/
$\Cotor^k_{A}(\k,M)$ is a unipotent $B$-comodule.
\end{prop}
\begin{proof}
There is a Cartan-Eilenberg spectral sequence
with
\[
\mathrm{E}_2^{s,t} =
\Cotor_{B}^s(\k,\Cotor_{A/\!/B}^t(\k,M))
\Lra \Cotor_{A}^{s+t}(\k,M).
\]
When $B$ is cocommutative this is a first
quadrant cohomological spectral sequence
of left $B$-comodules, and the differentials
are homomorphisms of unipotent $B$-comodules;
this is verified by noting that
the functor $\Cotor_{A/\!/B}^t(\k,-)\iso\Cotor_{A}^t(B,-)$
takes values in the category of left $B$-comodules,
and when~$B$ is cocommutative and~$U$ is
finite dimensional,~$\Cohom_B(U,V)$ is
naturally a left comodule. By Lemma~\ref{lem:unipotent-2from3},
each $\mathrm{E}_\infty^{s,t}$ is a unipotent
$B$-comodule and the same is true of
each~$\Cotor_{A}^{n}(\k,M)$.
\end{proof}

In order to state a special case which will
be used later, we  recall a standard result
which can be found
in~\cite{JPM&KP:MoreConcise}*{corollary~21.24}.
Let $B\subseteq A$ and $C\subseteq A$ be
conormal subHopf algebras of a commutative
Hopf algebra over a field\/~$\k$ with
$C\subseteq B$. Then there is an induced
normal inclusion $B/\!/C\to A/\!/B$ and
an isomorphism of Hopf algebras
\begin{equation}\label{eq:(A/C)/(B/C)=A/B}
(A/\!/C)/\!/(B/\!/C)\iso A/\!/B.
\end{equation}
\begin{cor}\label{cor:Unipotent-Cotor-general}
Suppose that $B/\!/C$ is cocommutative and
$M$ is a left $A/\!/C$-comodule such
that for every~$k\geq0$,\/ $\Cotor^k_{A/\!/B}(\k,M)$
is a unipotent left $A/\!/C$-comodule. Then
each cohomology group\/ $\Cotor^k_{A/\!/C}(\k,M)$
is a unipotent $B/\!/C$-comodule.
\end{cor}
\begin{proof}
Use the Cartan-Eilenberg spectral sequence
in the proof of Proposition~\ref{prop:Unipotent-Cotor-general},
\[
\mathrm{E}_2^{s,t} =
\Cotor_{B/\!/C}^s(\k,\Cotor_{(A/\!/C)/\!/(B/\!/C)}^t(\k,M))
\Lra \Cotor_{A/\!/C}^{s+t}(\k,M).
\]
where by \eqref{eq:(A/C)/(B/C)=A/B},
\[
\Cotor_{(A/\!/C)/\!/(B/\!/C)}^t(\k,M)
\iso \Cotor_{A/\!/B}^t(\k,M).
\qedhere
\]
\end{proof}

Our motivation for developing these ideas
is the following result.

\begin{prop}\label{prop:Unipotent-Cohom}
Suppose that $A_*$ is a $P_*$-algebra and
$M_*$ a unipotent left $A_*$-comodule. Then
\[
\Cohom_{A_*}(A_*,M_*) = 0.
\]
\end{prop}
\begin{proof}
A non-trivial comodule homomorphism $f\:A_*\to M_*$
factors through some subcomodule $M_*^k\subseteq M_*$
in a unipotent filtration where
and~$k$ is minimal. Now choose a non-zero
element $x\in M_*^k\cap\im f$ and choose
a $\k$-linear map $M_*^k/M_*^{k-1}\to\k$
so that $x\mapsto 1$ under the composition.
The resulting composition
\[
\xymatrix{
A_*\ar[r]_f\ar@/^15pt/[rrr] & M_*^k\ar[r]
& M_*^k/M_*^{k-1}\ar[r] & \k
}
\]
is a non-trivial comodule homomorphism
which cannot exist by
Proposition~\ref{prop:P*alg-finite}.
\end{proof}
\begin{cor}\label{cor:Unipotent-Cohom}
For $s\geq0$,
\[
\Coext^s_{A_*}(A_*,M_*) = 0.
\]
\end{cor}
\begin{proof}
This can be proved by induction of the
length of a unipotent filtration for~$M_*$.
\end{proof}

We will use this repeatedly in what follows
to show that certain~$\Coext$ groups vanish.

\section{Recollections on the Steenrod algebra
and its dual}\label{sec:StA}

The theory of \Palgebras{} applies to many
situations involving sub and quotient Hopf
algebras of the Steenrod algebra and its
dual for a prime. We will focus attention
on the prime~$2$ but the methods are
applicable for all primes.

To illustrate this, here is a simple application
involving the mod~$2$ Steenrod algebra; this
result appears in~\cite{DCR:Localn}*{corollary~4.10}.
We denote the mod~$2$ Eilenberg-Mac~Lane spectrum
by~$H=H\F_2$ and recall that
\[
\StA_*
= H_*(H)
= \F_2[\xi_1,\ldots,\xi_n,\ldots]
= \F_2[\zeta_1,\ldots,\zeta_n,\ldots]
\]
where $\zeta_n=\chi(\xi_n)\in\StA_{2^n-1}$
and the coproduct is given by the equivalent
formulae
\[
\psi\xi_n =
\sum_{0\leq i\leq n}\xi_{n-i}^{2^{i}}\otimes \xi_{i},
\quad
\psi\zeta_n =
\sum_{0\leq i\leq n}\zeta_{i}\otimes\zeta_{n-i}^{2^{i}}.
\]

\subsection*{Doubling}
%\label{sec:Doubling}

The operation of \emph{doubling} has been
used frequently in studying $\StA$-modules.
The reader is referred to the account of
Margolis~\cite{HRM:Book}*{section~15.3}
which we will use as background.

Since the dual $\StA_*$ is a commutative
Hopf algebra, it admits a Frobenius
endomorphism $\StA_*\to\StA_*$ which
doubles degrees and has Hopf algebra
cokernel
\[
\mathcal{E}_* = \StA_*/\!/\StA_*^{(1)}
= \Lambda_{\F_2}(\bar{\zeta}_s : s\geq1),
\]
where $\StA_*^{(1)}=\F_2[\zeta_s^2 : s\geq1 ]$.
Dually, there is a Verschiebung $\StA\to\StA$
which halves degrees and satisfies
\[
\Sq^r\mapsto
\begin{dcases*}
\Sq^{r/2} & if $r$ is even, \\
\ph{\;\;\;}0 & if $r$ is odd.
\end{dcases*}
\]
The kernel of this Verschiebung is the
ideal generated by the Milnor primitives
$\mathrm{P}^0_t$ ($t\geq1$), hence there
is a degree-halving isomorphism of Hopf
algebras $\StA/\!/\mathcal{E}\xrightarrow{\iso}\StA$,
where $\mathcal{E}\subseteq\StA$ is the
subHopf algebra generated by the primitives
$\mathrm{P}^0_t$ and dual to the exterior
quotient Hopf algebra $\mathcal{E}_*$.

Given a left (graded) $\StA$-module~$M$, we
can induce an $\StA/\!/\mathcal{E}$-module
$M_{(1)}$ where
\[
M_{(1)}^n =
\begin{dcases*}
M^{n/2} & if $n$ is even, \\
0 & if $n$ is odd,
\end{dcases*}
\]
and we write $x_{(1)}$ to indicate the
element $x\in M$ regarded as an element
of $M_{(1)}$; the module structure is
given by
\[
\Sq^r(x_{(1)}) =
\begin{dcases*}
(\Sq^{r/2}x)_{(1)} & if $r$ is even, \\
0 & if $r$ is odd,
\end{dcases*}
\]

Using this construction, the category of
left $\StA$-modules $\Mod_{\StA}$ admits
an additive functor to the category of
evenly graded $\StA/\!/\mathcal{E}$-modules,
\[
\Phi\:
\Mod_{\StA} \to
\Mod_{\StA/\!/\mathcal{E}}^\ev;
\quad
M\mapsto M_{(1)}
\]
which is an isomorphism of categories.
The quotient homomorphism $\rho\:\StA\to\StA/\!/\mathcal{E}$
also induces an additive isomorphism of categories
$\rho^*\:\Mod_{\StA/\!/\mathcal{E}}^\ev
\to\Mod_{\StA}^\ev$
and it is often useful to consider the
composition
$\rho^*\circ\Phi\:\Mod_{\StA}
\to \Mod_{\StA}^\ev$.

By iterating $\Phi^{(1)} = \Phi$ we obtain
isomorphisms
\[
\Phi^{(s)} = \Phi\circ\Phi^{(s-1)}\:\Mod_{\StA}
\to
\Mod_{\StA/\!/\StE^{(s-1)}}^{(s)};
\quad
M\mapsto M_{(s)}
\]
where the codomain is the category of
$\StA/\!/\StE^{(s-1)}$-modules concentrated
in degrees divisible by~$2^s$ and
$\StE^{(s-1)}\subseteq\StA$ is the subHopf
algebra multiplicatively generated by the
elements
\begin{equation}\label{eq:E(s)gens}
\mathrm{P}^a_b \quad (s-1\geq a\geq0,\;b\geq1),
\end{equation}
and $\StE^{(0)}=\StE$.

By doubling all three of the variables involved
the following homological result is immediate
for $e\geq1$ and two $\StA$-modules~$M,N$:
\begin{align}\label{eq:Doubling-iterated-Ext}
\Ext_{\StA_{(e)}}^{s,2^et}(M_{(e)},N_{(e)})
\iso \Ext_\StA^{s,t}(M,N).
\end{align}

Because doubling is induced using a grade
changing Hopf algebra endomorphism, the
double $\StA_{(1)}$ is also a Hopf algebra
isomorphic to the quotient Hopf algebra
$\StA/\!/\mathcal{E}$ and dual to the
subHopf algebra of squares
$\StA^{(1)}_*\subseteq\StA_*$ which is
also given by
\[
\StA^{(1)}_*
= \StA_*\square_{\StA_*/\!/\StA^{(1)}_*}\F_2
= \F_2\square_{\StA_*/\!/\StA^{(1)}_*}\StA_*
= (\StA_*/\!/\StA^{(1)}_*)\backslash\!\backslash\StA_*.
\]
More generally, for any $s\geq1$, $\StA_{(s)}$
is isomorphic to the quotient Hopf algebra
of~$\StA/\!/\StE^{(s)}$ dual to the subalgebra
of $2^s$-th powers
\[
\StA^{(s)}_*
= (\StA_*/\!/\StA^{(s)}_*)\backslash\!\backslash\StA_*
\subseteq\StA_*.
\]

In many ways, doubling is more transparent
when viewed in terms of comodules. For
an $\StA_*$-comodule $M_*$, we can define
a $\StA_*^{(1)}$-coaction
$\mu^{(1)}\colon M_*^{(1)}\to\StA_*^{(1)}\otimes M_*^{(1)}$
where $M_*^{(1)}$ denotes $M_*$ with its
degrees doubled; this is given on elements
by the composition
\[
\xymatrix{
M_*\ar[rr]_(.45)\mu \ar@/^15pt/[rrrr]^{\mu^{(1)}}
&& \StA_*\otimes M_*\ar[rr]_(.45){(-)^2\otimes\Id}
&& \StA_*^{(1)}\otimes M_*.
}
\]
By iterating we also obtain a $\StA_*^{(s)}$-coaction
$\mu^{(s)}\colon M_*^{(s)}\to\StA_*^{(s)}\otimes M_*^{(s)}$.

Then the comodule analogue of~\eqref{eq:Doubling-iterated-Ext}
is
\begin{align}\label{eq:Doubling-iterated-Coxt}
\Coext_{\StA_*^{(e)}}^{s,2^et}(M_*^{(e)},N_*^{(e)})
\iso \Coext_{\StA_*}^{s,t}(M_*,N_*).
\end{align}

We can use iterated doubling combined with
Proposition~\ref{prop:ExtA/B} to show that
for any $d\geq1$,
\begin{equation}\label{eq:Doubling}
\Coext_{\StA_*}^{s,t}(\StA_*,\StA_*^{(d)})
\iso \Ext_\StA^{s,t}(\StA_{(d)},\StA)
= 0.
\end{equation}
By doubling all three of the variables involved
here we can also prove that for $e\geq0$,
\begin{equation}\label{eq:Doubling-iterated-comodules}
\Coext_{\StA_*^{(e)}}^{s,2^et}(\StA_*^{(e)},\StA_*^{(d+e)})
\iso \Coext_{\StA_*}^{s,t}(\StA_*,\StA_*^{(d)})
= 0.
\end{equation}

\subsection*{Some families of quotient $P_*$-algebras
of $\StA_*$}
We will begin by describing some quotients
of the dual Steenrod algebra~$\StA_*$. For
any $n\geq1$, $(\zeta_1,\ldots,\zeta_n)\lhd\StA_*$
is a Hopf ideal so there is a quotient Hopf
algebra $\StA_*/(\zeta_1,\ldots,\zeta_n)$
together with the subHopf algebra
\[
\StP(n)_*
= \StA_*\square_{\StA_*/(\zeta_1,\ldots,\zeta_n)}\F_2
= \F_2[\zeta_1,\ldots,\zeta_n]\subseteq\StA_*
\]
and in fact
\[
\StA_*/\!/\StP(n)_* = \StA_*/(\zeta_1,\ldots,\zeta_n).
\]

Similarly, for any $s\geq0$, the ideal
$(\zeta_1^{2^s},\ldots,\zeta_n^{2^s})\lhd\StA_*$
is a Hopf ideal and there is a quotient
Hopf algebra
\[
\StA_*/\!/\StP(n)^{(s)}_*
= \StA_*/(\zeta_1^{2^s},\ldots,\zeta_n^{2^s})
\]
with associated subHopf algebra
\[
\StP(n)^{(s)}_*
= \StA_*\square_{\StA_*/\!/\StP(n)^{(s)}_*}\F_2
= \F_2[\zeta_1^{2^s},\ldots,\zeta_n^{2^s}]\subseteq\StA_*.
\]
For each $t\geq0$ there is a finite quotient
Hopf algebra
\[
\StP(n)^{(s)}_*/
(\zeta_1^{2^{s+t}},\zeta_2^{2^{s+t-1}},\ldots,
\zeta_t^{2^{s+1}},\zeta_{t+1}^{2^{s}},\ldots,\zeta_n^{2^s})
\]
and we have
\[
\StP(n)^{(s)}_* =
\lim_t \StP(n)^{(s)}_*/
(\zeta_1^{2^{s+t}},\zeta_2^{2^{s+t-1}},\ldots,
\zeta_t^{2^{s+1}},\zeta_{t+1}^{2^{s}},\ldots,\zeta_n^{2^s})
\]
where the limit is computed degree-wise. The
graded dual Hopf algebra
\[
\StP(n)_{(s)} =
(\StP(n)_*^{(s)})^*=\Hom(\StP(n)^{(s)}_*,\F_2)
\]
is the colimit of the finite dual Hopf algebras
\[
\Hom(\StP(n)^{(s)}_*/
(\zeta_1^{2^{s+t}},\zeta_2^{2^{s+t-1}},\ldots,
\zeta_t^{2^{s+1}},\zeta_{t+1}^{2^{s}},\ldots,\zeta_n^{2^s}),\F_2),
\]
i.e.,
\[
\StP(n)_{(s)} = \colim_t\Hom(\StP(n)^{(s)}_*/
(\zeta_1^{2^{s+t}},\zeta_2^{2^{s+t-1}},\ldots,
\zeta_t^{2^{s+1}},\zeta_{t+1}^{2^{s}},\ldots,\zeta_n^{2^s}),\F_2).
\]
Therefore $\StP(n)_{(s)}$ is a \Palgebra{} and
$\StP(n)^{(s)}_*$ is a $P_*$-algebra.

%There are primitive elements $\MP_k^{(s)}\in\StP(n;s)^*$,
%the duals of the generators~$\zeta_k^{2^s}$ for
%$1\leq k\leq n$. These are primitives which also
%satisfy
%\[
%(\MP_k^{(s)})^2 = 0.
%\]
%Now $\MP_k^{(s)}$ acts on any $\StP(n;s)_*$-comodule
%or $\StP(n;s)^*$-module, and there are associated
%Margolis (co)homology functors $\mathrm{H}(-;\MP_k^{(s)})$.
%In particular,
%\[
%\mathrm{H}(\StP(n;s)_*;\MP_k^{(s)})
%=0=
%\mathrm{H}(\StP(n;s)^*;\MP_k^{(s)}).
%\]
%%By \cite{HRM:Book}*{15.1.4}

\section{Some comodules and their cohomology}
\label{sec:Cotor}

Later we will need to determine various
cohomology groups such as
\[
\Coext^{*,*}_{\StA_*}(\StA^{(1)}_*,\F_2).
\]
using the Cartan-Eilenberg spectral sequence
of Proposition~\ref{prop:CESS-1},
\begin{equation}\label{eq:CESS-CoextA(1)}
\mathrm{E}^{s,t}_2 =
\Coext^s_{\StA^{(1)}_*}(\StA^{(1)}_*,\Cotor^t_{\StA_*/\!/\StA^{(1)}_*}(\F_2,\F_2))
\Lra \Coext^{s+t}_{\StA_*}(\StA^{(1)}_*,\F_2),
\end{equation}
where we have suppressed the internal grading.
In fact, $\StA^{(1)}_*$ is a projective
$\StA^{(1)}_*$-comodule, so $\mathrm{E}^{s,t}_2 = 0$
when $s>0$, therefore we only need to consider
\[
\mathrm{E}^{0,t}_2 =
\Cohom_{\StA^{(1)}_*}(\StA^{(1)}_*,
\Cotor^t_{\StA_*/\!/\StA^{(1)}_*}(\F_2,\F_2)).
\]
Here the $\Cotor$ term is bigraded with
\[
\Cotor^{*,*}_{\StA_*/\!/\StA^{(1)}_*}(\F_2,\F_2)
= \F_2[ q_n : n\geq0 ]
\]
for $q_n=[\bar{\zeta_{n+1}}]\in\Cotor^{1,2^{n+1}-1}$
represented in the cobar construction by the
residue class of
$\bar{\zeta_{n+1}}\in\StA_*/\!/\StA^{(1)}_*$.
We need to understand the comodule structure
on this and similar~$\Cotor$ groups.

Now we can consider the adjoint coaction
for~$\StA_*$.
\begin{lem}\label{lem:AdjCoaction}
The left adjoint coaction of~$\StA_*$ is
given by
\begin{equation*}%\label{eq:AdjCoaction}
\mu\:\StA_*\to \StA_*\otimes\StA_*;
\quad
\mu(\zeta_n) =
\sum_{i\geq0}\sum_{j\geq0}\zeta_i\xi_{n-i-j}^{2^{i+j}}\otimes\zeta_j^{2^i}.
\end{equation*}
\end{lem}
\begin{proof}
This follows by iterating the coaction
and using the formulae
\[
\psi\zeta_n =
\sum_{0\leq k\leq n}\zeta_k\otimes\zeta_{n-k}^{2^k},
\quad
\xi_r  =\chi(\zeta_r).
\qedhere
\]
\end{proof}

The left coaction on~$q_n$ can be deduced from
that on~$\zeta_{n+1}$ where we can ignore
all terms in the sum for $\mu(\zeta_{n+1})$
with $i>0$, thus giving
\begin{equation}\label{eq:Coaction-q_k}
\mu q_n
=
\sum_{0\leq j\leq n}\xi_{n-j}^{2^{j+1}}\otimes q_j.
\end{equation}
This extends to polynomials in the~$q_n$ using
multiplicativity.

This coaction is visibly defined over
$\StA^{(1)}_*\subseteq\StA_*$, and in
practise we will only require the coaction over
the quotient $\StA^{(1)}_*/\!/\StA^{(3)}_*$ so
we will denote this coaction by
\[
\bar{\mu}\:
\Cotor^{*,*}_{\StA_*/\!/\StA^{(1)}_*}(\F_2,\F_2)
\to
\StA^{(1)}_*/\!/\StA^{(3)}_*\otimes
\Cotor^{*,*}_{\StA_*/\!/\StA^{(1)}_*}(\F_2,\F_2),
\]
where for $n\geq2$,
\[
\bar{\mu}q_n =
\bar{\xi_n^2}\otimes q_0 + \bar{\xi_{n-1}^4}\otimes q_1 + 1\otimes q_n
\]
and
\[
\bar{\mu}q_1 = \bar{\xi_1^2}\otimes q_0 + 1\otimes q_1.
\]
In what follows we will make use of this
$\StA^{(1)}_*/\!/\StA^{(3)}_*$-comodule structure
and also the induced $\StA^{(1)}_*/\!/\StA^{(2)}_*$-comodule
structure.

\begin{prop}\label{prop:A(1)->Cotor}
For $k\geq0$ there are no non-trivial
$\StA^{(1)}_*/\!/\StA^{(2)}_*$-comodule homomorphisms
\[
\StA^{(1)}_*\to
\Cotor^{k,*}_{\StA_*/\!/\StA^{(1)}_*}(\F_2,\F_2).
\]
Hence
\[
\Cohom_{\StA^{(1)}_*}(\StA^{(1)}_*,
\Cotor^{k,*}_{\StA_*/\!/\StA^{(1)}_*}(\F_2,\F_2))
=
\Cohom_{\StA^{(1)}_*/\!/\StA^{(2)}_*}(\StA^{(1)}_*,
\Cotor^{k,*}_{\StA_*/\!/\StA^{(1)}_*}(\F_2,\F_2))
=0.
\]
\end{prop}

Before giving the proof, we will define for
each $k\geq0$, an increasing filtration of
the $\StA^{(1)}_*/\!/\StA^{(2)}_*$-comodule
\[
\Cotor^{k,*}_{\StA_*/\!/\StA^{(1)}_*}(\F_2,\F_2)
=
\F_2\{ q_0^{r_0}q_1^{r_1}\cdots q_\ell^{r_\ell} :
\sum_{0\leq i\leq\ell}r_i = k \}
\]
by setting
\[
\mathrm{F}^{k,s} =
\F_2\{ q_0^{r_0}q_1^{r_1}\cdots q_\ell^{r_\ell} :
r_0\geq k-s,\; \sum_{0\leq i\leq\ell}r_i = k\}.
\]
Each $\mathrm{F}^{k,s}$ is a subcomodule of
$\Cotor^{k,*}_{\StA_*/\!/\StA^{(1)}_*}(\F_2,\F_2)$
and
\[
\mathrm{F}^{k,0} = \F_2\{ q_0^k \},\
\quad
\mathrm{F}^{k,k} =
\F_2\{ q_1^{r_1}\cdots q_\ell^{r_\ell} :
\sum_{0\leq i\leq\ell}r_i = k \}.
\]

Now suppose that
$f\:M_*\to\Cotor^{k,*}_{\StA_*/\!/\StA^{(1)}_*}(\F_2,\F_2)$
is a non-trivial comodule homomorphism. Choose
$s_0$ to be minimal so that $\im f\subseteq\mathrm{F}^{k,s_0}$.
Then the composition
\begin{equation}\label{eq:FiltnCollapse}
\xymatrix{
M_* \ar[r]_f\ar@/^15pt/[rr]^{\bar{f}}
& \mathrm{F}^{k,s_0}\ar[r] & \mathrm{F}^{k,s_0}/\mathrm{F}^{k,s_0-1}
}
\end{equation}
is a non-trivial comodule homomorphism, where
$\mathrm{F}^{k,s_0}/\mathrm{F}^{k,s_0-1}$
has the trivial coaction. Now we may choose
any non-trivial linear map
$\mathrm{F}^{k,s_0}/\mathrm{F}^{k,s_0-1}\to\F_2[d]$
which is non-trivial on $\im\bar{f}$ and so
obtain a non-trivial comodule homomorphism
$M_*\to\F_2[d]$. This is the key observation
required for our proof.

\begin{proof}[Proof of \emph{Proposition~\ref{prop:A(1)->Cotor}}]
Such a comodule homomorphism leads to a
non-trivial $\StA_*^{(1)}/\!/\StA_*^{(2)}$-comodule
homomorphism $\StA^{(1)}_*\to\F_2[d]$ for
some~$d$. By the Milnor-Moore theorem,
$\StA^{(1)}_*$ is an extended
$\StA_*^{(1)}/\!/\StA_*^{(2)}$-comodule,
so there must be a non-trivial
$\StA_*^{(1)}/\!/\StA_*^{(2)}$-comodule
homomorphism $\StA_*^{(1)}/\!/\StA_*^{(2)}\to\F_2[d']$
for some~$d'$. But since $\StA_*^{(1)}/\!/\StA_*^{(2)}$
is a $P_*$-algebra, this contradicts
Proposition~\ref{prop:CoherentComod-finite}.

Since the natural homomorphism
\[
\Cohom_{\StA^{(1)}_*}(\StA^{(1)}_*,
\Cotor^{k,*}_{\StA_*/\!/\StA^{(1)}_*}(\F_2,\F_2))
\to
\Cohom_{\StA^{(1)}_*/\!/\StA^{(2)}_*}(\StA^{(1)}_*,
\Cotor^{k,*}_{\StA_*/\!/\StA^{(1)}_*}(\F_2,\F_2))
\]
is injective, the last statement follows.
\end{proof}
\begin{cor}\label{cor:A(1)->Cotor}
The\/ $\mathrm{E}_2$-term of the spectral
sequence~\eqref{eq:CESS-CoextA(1)} is
trivial, hence
\[
\Coext^{*,*}_{\StA_*}(\StA^{(1)}_*,\F_2) = 0.
\]
\end{cor}

A generalisation of these results is
\begin{prop}\label{prop:A(1)->P(n)(1)}
For any $n\geq0$,
\[
\Coext^{*,*}_{\StA_*}(\StA^{(1)}_*,\StP(n)^{(1)}_*) = 0.
\]
\end{prop}
\begin{proof}
We need to deal with the case $n\geq1$.
By Proposition~\ref{prop:CESS-1} there
is a Cartan-Eilenberg spectral sequence
of form
\[
\mathrm{E}^{s,t}_2 =
\Coext^s_{\StA^{(1)}_*}(\StA^{(1)}_*,
\Cotor^t_{\StA_*/\!/\StA^{(1)}_*}(\F_2,\StP(n)^{(1)}_*))
\Lra \Coext^{s+t}_{\StA_*}(\StA^{(1)}_*,\F_2).
\]
Since the $\StA_*/\!/\StA^{(1)}_*$-coaction
on $\StP(n)^{(1)}_*$ is trivial,
\[
\Cotor^t_{\StA_*/\!/\StA^{(1)}_*}(\F_2,\StP(n)^{(1)}_*)
\iso
\Cotor^t_{\StA_*/\!/\StA^{(1)}_*}(\F_2,\F_2)
\overset{\StA^{(1)}_*}\wedge \StP(n)^{(1)}_*.
\]
As left $\StA^{(1)}_*$-comodules
\[
\StP(n)^{(1)}_* \iso
\StA^{(1)}_*\square_{\StA^{(1)}_*/\!/\StP(n)^{(1)}_*}\F_2
\]
so by a standard `change of rings' isomorphism,
\[
\mathrm{E}^{s,t}_2 \iso
\Coext^s_{\StA^{(1)}_*/\!/\StP(n)^{(1)}_*}(\StA^{(1)}_*,
\Cotor^t_{\StA_*/\!/\StA_*^{(1)}}(\F_2,\F_2)).
\]
Again the Milnor-Moore theorem shows that
$\StA^{(1)}_*$ is a cofree
$\StA^{(1)}_*/\!/\StP(n)^{(1)}_*$-comodule
and since this is a $P_*$-algebra,
$\mathrm{E}^{s,t}_2=0$ for $s>0$. Now
a similar argument to that in the proof
of Proposition~\ref{prop:A(1)->Cotor}
shows that~$\mathrm{E}^{0,t}_2=0$.
\end{proof}

Before passing on to discuss other examples,
we note that in~$\StA_*$ there is a subcomodule
$\leftidx{^{\leq k}}{\StA}{_*}$ spanned by the
monomials in the~$\xi_i$ of degree at most~$k$.
Under the doubling isomorphism this corresponds
to a subcomodule $\leftidx{^{\leq k}}{\StA}{_*^{(1)}}$
of $\StA^{(1)}_*$ spanned by monomials in the 
$\xi_i^2$ of degree at most~$k$.

\begin{prop}\label{prop:Cotor-comodule}
For $k\geq0$ there is an isomorphism of
$\StA^{(1)}_*$-comodules
\[
\Cotor^{k,*}_{\StA_*/\!/\StA^{(1)}_*}(\F_2,\F_2)
\xrightarrow{\iso}
\leftidx{^{\leq k}}{\StA}{_*^{(1)}};
\quad
q_0^{r_0}q_1^{r_1}\cdots q_\ell^{r_\ell}
\leftrightarrow
\xi_1^{2r_1}\cdots \xi_\ell^{2r_\ell}.
\]
\end{prop}

The following result is an analogue of
Proposition~\ref{prop:A(1)->Cotor} whose
proof can be adapted using the filtration
of $\leftidx{^{\leq k}}{\StA}{_*}$ based
on polynomial degree.
\begin{prop}\label{prop:A->Cotor}
For $k\geq0$,
\[
\Coext^*_{\StA_*}(\StA_*,\leftidx{^{\leq k}}{\StA}{_*})
=0.
\]
\end{prop}

We will also require some other vanishing results.
\begin{prop}\label{prop:A(1)->A(2)}
For $n\geq0$,
\[
\Coext^{*,*}_{\StA_*}(\StA^{(1)}_*,\StP(n)^{(2)}_*) = 0
\]
and
\[
\Coext^{*,*}_{\StA_*}(\StA^{(1)}_*,\StA^{(2)}_*) = 0.
\]
\end{prop}
\begin{proof}
By setting $\StP(\infty)^{(2)}_*=\StA^{(2)}_*$
we can present the proofs of these in a uniform
faahion.

There is a Cartan-Eilenberg spectral sequence
\[
\mathrm{E}^{s,t}_2 =
\Coext^s_{\StA^{(1)}_*}(\StA^{(1)}_*,
\Cotor^t_{\StA_*/\!/\StA^{(1)}_*}(\F_2,\StP(n)^{(2)}_*))
\Lra \Coext^{s+t}_{\StA_*}(\StA^{(1)}_*,\StP(n)^{(2)}_*).
\]
Here
\[
\Cotor^t_{\StA_*/\!/\StA^{(1)}_*}(\F_2,\StP(n)^{(2)}_*)
\iso
\Cotor^t_{\StA_*/\!/\StA^{(1)}_*}(\F_2,\F_2)
\overset{\StA^{(1)}_*}\wedge\StP(n)^{(2)}_*
\]
and
\[
\StP(n)^{(2)}_* \iso
\StA^{(1)}_*\square_{\StA^{(1)}_*/\!/\StP(n)^{(2)}_*}\F_2
\]
so
\[
\mathrm{E}^{s,t}_2 \iso
\Coext^s_{\StA^{(1)}_*/\!/\StP(n)^{(2)}_*}(\StA^{(1)}_*,
\Cotor^t_{\StA_*/\!/\StA_*^{(1)}}(\F_2,\F_2)).
\]
Since $\StA^{(1)}_*$ is a cofree
$\StA^{(1)}_*/\!/\StP(n)^{(2)}_*$-comodule,
$\mathrm{E}^{s,t}_2=0$ when $s>0$. Also the
change of coalgebra homomorphism
\[
\Cohom_{\StA^{(1)}_*/\!/\StP(n)^{(2)}_*}(\StA^{(1)}_*,
\Cotor^t_{\StA_*/\!/\StA_*^{(1)}}(\F_2,\F_2))
\to
\Cohom_{\StA^{(1)}_*/\!/\StA^{(2)}_*}(\StA^{(1)}_*,
\Cotor^t_{\StA_*/\!/\StA_*^{(1)}}(\F_2,\F_2))
\]
is injective. By Proposition~\ref{prop:A(1)->Cotor},
the codomain is trivial so $\mathrm{E}^{0,t}_2=0$.
%
%There is a Cartan-Eilenberg spectral
%sequence
%\[
%\mathrm{E}^{s,t}_2 =
%\Coext^s_{\StA^{(1)}_*}(\StA^{(1)}_*,
%\Cotor^t_{\StA_*/\!/\StA^{(1)}_*}(\F_2,\StA^{(2)}_*))
%\Lra \Coext^{s+t}_{\StA_*}(\StA^{(1)}_*,\StA^{(2)}_*).
%\]
%Here
%\[
%\Cotor^t_{\StA_*/\!/\StA^{(1)}_*}(\F_2,\StA^{(2)}_*)
%\iso
%\Cotor^t_{\StA_*/\!/\StA^{(1)}_*}(\F_2,\F_2)
%\overset{\StA^{(1)}_*}\wedge\StA^{(2)}_*
%\]
%and
%\[
%\StA^{(2)}_* \iso
%\StA^{(1)}_*\square_{\StA^{(1)}_*/\!/\StA^{(2)}_*}\F_2
%\]
%so
%\[
%\mathrm{E}^{s,t}_2 \iso
%\Coext^s_{\StA^{(1)}_*/\!/\StA^{(2)}_*}(\StA^{(1)}_*,
%\Cotor^t_{\StA_*/\!/\StA_*^{(1)}}(\F_2,\F_2)).
%\]
%Arguing in a similar way to the previous results,
%we find that $\mathrm{E}^{s,t}_2=0$.
\end{proof}

\section{Some topological applications}\label{sec:TopApplns}

We illustrate our theory with a few calculations
of homotopy groups of spectra. Some of the
following examples where found in response
to questions of John Rognes. Results of this
type were proved by Lin, Margolis, Ravenel
and others. The interested reader will be 
able to give others, especially at odd primes.

For simplicity we assume that all spectra
are $2$-completed. We are interested in
determining when $[X,Y]^*=0$ for two
connective finite type spectra~$X,Y$. Here
the Adams spectral sequence
\[
\mathrm{E}_2^{s,t}(X,Y)
=\Ext_{\StA}^{s,t}(H^*(Y),H^*(X))
=\Coext_{\StA}^{s,t}(H_*(X),H_*(Y))
\Lra [X,Y]^{s-t}
\]
converges by work of Boardman~\cite{JMB:CCSS},
so we are interested in examples where
$\mathrm{E}_2^{*,*}(X,Y)=0$. The general
result of Proposition~\ref{prop:Finite->coherent}
applies to many interesting examples.

\begin{itemize}
\item
For each the spectra $X$ where $H=H\F_2$,
$H\Z_{2^s}$ ($s\geq2$), $H\Z$, $\kO$,
$\kU$, $\tmf$, and $BP\langle n\rangle$
($n\geq1$), $[X,S]^*=0$ since $H^*(X)$
is a coherent $\StA$-module and
$\mathrm{E}_2^{*,*}(X,S)=0$.
\item
For each of the above spectra~$X$,
$[X,BP]^*=0$. This follows since
\[
H^*(BP) \iso \StA\otimes_{\StE}\F_2
\]
where $\StE\subseteq\StA$ is the
subHopf algebra generated by the
Milnor primitives and this is
also a $P$-algebra. So by
Proposition~\ref{prop:IndFiniteB->coherent},
\[
\mathrm{E}_2^{*,*}(X,BP) \iso
\Ext_{\StE}^{*,*}(\F_2,H^*(X))=0.
\]
\item
The case of $[BP,S]^*$ involves more work.
We will use the dual version of the Adams
spectral sequence,
\[
\mathrm{E}_2^{*,*}(BP,S)
= \Coext_{\StA_*}^{s,t}(\StA^{(1)}_*,\F_2)
\Lra [BP,S]^{s-t}.
\]
We can calculate the $\mathrm{E}_2$-term
using the Cartan-Eilenberg spectral
sequence of~\eqref{eq:CESS-CoextA(1)},
and by Corollary~\ref{cor:A(1)->Cotor}
this $\mathrm{E}_2$-term
\[
\Coext^*_{\StA^{(1)}_*}(\StA^{(1)}_*,\Cotor^*_{\StA_*/\!/\StA^{(1)}_*}(\F_2,\F_2))
\]
is trivial. Therefore $\mathrm{E}_2^{*,*}(BP,S)=0$
and~$[BP,S]^*=0$.
\end{itemize}

\end{document}